\documentclass[a4paper, 10pt]{amsart}
\usepackage{amssymb}

\usepackage{amsmath,amssymb,amsthm}
\usepackage[latin1]{inputenc}
\usepackage{version,tabularx,multicol}
\usepackage{graphicx,float,psfrag}
\usepackage{stmaryrd}
 \usepackage{color}
\usepackage{latexsym, bbding}

\theoremstyle{plain}
\newtheorem{theorem}{Theorem}[section]

 \newtheorem{corollary}[theorem]{Corollary}

 \newtheorem{lemma}[theorem]{Lemma}

 \newtheorem{remark}[theorem]{Remark}

\usepackage{mathrsfs}

\numberwithin{equation}{section}

 \def\qed{\hfill$\Box$\medskip}

 \def\mbb{\mathbb}

 \def\<{\langle}\def\>{\rangle}

 \def\beqlb{\begin{eqnarray}}\def\eeqlb{\end{eqnarray}}
 \def\beqnn{\begin{eqnarray*}}\def\eeqnn{\end{eqnarray*}}

 \def\mbb{\mathbb}

 \def\<{\langle}\def\>{\rangle}

 \def\beqlb{\begin{eqnarray}}\def\eeqlb{\end{eqnarray}}
 \def\beqnn{\begin{eqnarray*}}\def\eeqnn{\end{eqnarray*}}
  
\title[Harmonic moments and Large deviations]
 {Harmonic moments and large deviations for a critical Galton-Watson process with immigration}

\date{\today}
\address{School of Mathematical Sciences, Beijing Normal University, Beijing 100875, People's Republic of China.}
 \author{Doudou Li, Mei Zhang$^\ast$ }
\email{lidoudou@mail.bnu.edu.cn, \Envelope meizhang@bnu.edu.cn(*Corresponding
author)}
\keywords{harmonic moment, large deviation, local probability estimate, immigration.\\
Supported by NSFC (No. 11871103)
}

\begin{document}

\begin{abstract}
  In this paper, a critical Galton-Watson branching process with immigration $Z_{n}$ is studied. We first obtain the convergence rate of the harmonic moment of $Z_{n}$. Then the large deviation of  $S_{Z_n}:=\sum_{i=1}^{Z_n} X_i$ is obtained, where $\{X_i\}$ is a sequence of independent and identically distributed zero-mean random variables with tail index $\alpha>2$. We shall see that the converging rate is determined by the immigration mean, the variance of reproducing and the tail index  of $X_1^+$, comparing to  previous result for supercritical case, where the rate depends on the Schr\"{o}der constant and the tail index.
\end{abstract}

\maketitle

\section{Introduction and Main results}

Suppose $\{\xi_{ni},n,i\geq 1\}$ is a sequence of non-negative integer-valued independent and identically distributed (i.i.d.) random variables with  generating function $f(s)=\sum_{i=0}^\infty p_is^i$. \{$Y_n,n\geq 1$\} is another sequence of non-negative integer-valued i.i.d. random variables with  generating function $h(s)=\sum_{i=0}^\infty h_is^i$. $\{\xi_{ni},~n,i\geq1\}$ are independent of \{$Y_n,n\geq 1$\}. Define $\{Z_{n}\}$ recursively as
\begin{align}\label{1.1}
Z_{n}=\sum_{i=1}^{Z_{n-1}}\xi_{ni}+Y_n,~~n\geq 1,\quad Z_{0}= 0.
\end{align}
$\{Z_n, n\ge 0\}$ is called a Galton-Watson branching process with immigration (GWI). Denote $m:=\mbb{E}\xi_{11}$. When $m>1, m=1$ or $m<1$, we shall refer to $\{Z_n\}$ as supercritical, critical and subcritical, respectively. In this paper, we will study the critical case. From \eqref{1.1}, the generating function of $Z_n$ can be expressed by
\begin{align}\label{1.2}
H_{n}(x)=\prod_{k=0}^{n-1}h[f_{k}(x)],\quad n\ge 1,
\end{align}
where $f_{k}(x)$ denotes the $k$th iteration of the function $f(x)$ and $f_{0}(x)=x$. Denote $\beta=h'(1)$ and $\gamma=\frac{1}{2}f''(1)$~(where $h'(s):=\sum_{i=1}^\infty ih_is^{i-1}$, $f''(s):=\sum_{i=2}^\infty i(i-1)p_is^{i-2},~|s|<1)$. Define  $\sigma=\frac{\beta}{\gamma}$. In the present paper, we suppose the following condition holds:

(A) $g.c.d.\{k:p_{k}>0\}=1,~0< p_{0},h_{0}< 1,~m=1,~0<\beta<\infty,~0< \gamma< \infty$, $\sum\limits_{j=1}^{\infty}p_{j}j^{2}\log j< \infty$,~ $\sum\limits_{j=1}^{\infty}h_{j}j^{2}<\infty$.

Define the harmonic moment of $Z_{n}$ by
\begin{align*}
J_{n}(r)=\mbb{E}[Z_{n}^{-r}|Z_{n}>0],\quad r>0.
\end{align*}
   There have been some research works on the asymptotic behavior of $J_{n}(r)$ for the branching processes with or without immigration. For the supercritical case, we can refer to \cite{H71,N67,N03,P75,S17}. For the critical case, Nagaev~\cite{N67} investigated the asymptotic behavior of $J_{n}(1)$ for the Galton-Watson process without immigration. Later, Pakes~\cite{P75} refined  the  results of \cite{N67} under further moment assumption and considered the immigration. Recently, Li and Zhang~\cite{18} proved the conjecture that $J_{n}(1)\sim \frac{\log n}{\gamma n}(\sigma=1)$ in \cite{P75}, and obtained the convergence rate for $J_{n}(2)$.

 In the first part of this paper, we shall  study  the limit of $J_{n}(r)$ for the critical Galton-Watson branching process with immigration $Z_n$ defined by  \eqref{1.1}. We can see that there is  a phase transition: For $r>\sigma$, $r=\sigma$~or $r<\sigma$,  $J_{n}(r)$ has different asymptotic behaviors; see Theorem~\ref{t1.1}. Our main technical tool in the proofs is precise estimate of  generating function $H_{n}$. This theorem improves our previous result in \cite{18}.

By  \eqref{1.1},
$$
I_{\{Z_{n}>0\}}\frac{Z_{n+1}-mZ_{n}}{Z_n}\stackrel{d}=
I_{\{Z_{n}>0\}}\left(\frac{1}{Z_n}\sum_{i=1}^{Z_n}X_i+\frac{Y_{n+1}}{Z_n}\right),$$
where $\{X_i\}$ are i.i.d. variables,  $X_{1}\stackrel{d}=\xi_{11}-m$, and $\{X_i\}$ are independent of $Z_n$ and $Y_{n+1}$. We have known that as $n\to \infty$, $Z_{n+1}/Z_n\xrightarrow{ {P}}m$ (see \cite {AN72} and \cite {LZ16}). There have been some research works on the converging rate of $Z_{n+1}/Z_n$ (\cite{A94,18,N03}). To study the large deviations for $Z_{n+1}/Z_n$,
harmonic moment  plays a significant role, see for example, \cite{18,N03}.

Now let us see a general case. Let $S_{n}=\sum_{i=1}^{n}X_{i}$,  where $\{X_{i},i\geq 1\}$ are i.i.d. random variables, independent of $Z_{n}$ and satisfy
\begin{align}\label{1.3}
\mbb{E}X_{1}=0,\quad \sigma_{0}^{2}:=\mbb{E}X_{1}^{2}\in(0,\infty).
\end{align}
Define $X_{1}^{+}=X_{1}\vee 0$. We say that $X_{1}^{+}$ has a tail of index $\alpha$, if
\begin{align}\label{1.6}
\mbb{P}(X_{1}^{+}\geq x)\sim ax^{-\alpha},\quad x\rightarrow\infty,
\end{align}
for some  $a>0$.  Define
\begin{align}\label{1.4}
L_n=\frac{S_{Z_{n}}}{Z_{n}}I_{\{Z_{n}>0\}}.
\end{align}
By Pakes~\cite{P72}, $\mbb{P}\{Z_n>0\}\to 1$, so $L_n$ is well-defined. We can prove that (see \cite[Theorem 3]{N67} for more detail)
$$
\mbb{P}(\sqrt{n}L_n\le x)\to \phi(x),
$$
where $$\int_{0}^{\infty}e^{itx}d\phi(x)=
\frac{1}{\Gamma(\sigma)\gamma^{\sigma}}\int_{0}^{\infty}e^{-\frac{\sigma_{0}^{2} t^{2}}{2y}}y^{\sigma-1}e^{-\frac{y}{\gamma}}dy,$$
and $\phi(x)\rightarrow1$ as $x\rightarrow\infty$. Therefore, if  $\varepsilon_{n}\rightarrow 0 $ and $n\varepsilon^{2}_{n}\rightarrow\infty$, then

\begin{align*}
\mbb{P}(L_{n}\geq \varepsilon_{n})\to 0,  \quad n\rightarrow\infty.
\end{align*}
  We are interested in the rate of such convergence. If $X_{1}\stackrel{d}=\xi_{11}-m$, then  $L_{n}$ has the same limit behavior as  $\frac{Z_{n+1}}{Z_n}-m$. When  $Z_n$ is a supercritical Galton-Watson process (without immigration), the large deviation of $L_n$ was considered in \cite{N04} and \cite{F07}.

In the second part of the paper, we focus on the large deviations of $L_n$, i.e., the convergence rate of $\mbb{P}(L_{n}\geq \varepsilon_{n})$, where $Z_n$ is a critical Galton-Watson branching process with immigration defined by  \eqref{1.1}, $\{X_{i},i\geq 1\}$ are i.i.d. random variables  satisfying \eqref{1.3}, and  $\{\varepsilon_{n}\}$ is a sequence of positive numbers such that
\begin{align}\label{1.5}
\varepsilon_{n}\rightarrow 0,\quad \mbox{and}\quad  n\varepsilon^{2}_{n}\rightarrow\infty.
\end{align}
The converging rate depends on the tail index of $X_1^+$ and $\sigma$; see Theorem~\ref{t1.2}.

 Here are our main results.

\begin{theorem}\label{t1.1}
Suppose  condition (A) is satisfied and $\sigma>0$. For $r>0$, define
\begin{align}\label{1.2.1}
A(n,r)=
\begin{cases}
n^{\sigma},& r>\sigma,  \\
\frac{n^{\sigma}}{\log n},& r=\sigma, \\
n^{r},& r<\sigma.
\end{cases}
\end{align}
Then we have
\begin{align}\label{1.2.2}
\lim\limits_{n\rightarrow\infty}A(n,r)\mbb{E}(Z_{n}^{-r}|Z_{n}>0)=I(r,\sigma),
\end{align}
where
\begin{align*}
I(r,\sigma):=
\begin{cases}
\frac{1}{\Gamma(r)}\int_{0}^{\infty}(U(e^{-t})-U(0))t^{r-1}dt
,& \mbox{if}~r>\sigma,  \\
\frac{1}{\Gamma(r)}\gamma^{-\sigma},& \mbox{if}~r=\sigma, \\
\frac{1}{\Gamma(r)}\int_{0}^{\infty}(1+\gamma s)^{-\sigma}s^{r-1}ds,& \mbox{if}~r<\sigma,
\end{cases}
\end{align*}
and $U(\cdot)$ is defined as in Lemma \ref{l2.1}.
\end{theorem}

 Let
\begin{align}\label{1.7}
\rho=\frac{1+\sigma-\alpha}{2\sigma-\alpha}.
\end{align}

\begin{theorem}\label{t1.2}
Suppose  condition (A) is satisfied and $\sigma>1$. \\
(a)~If
$
\mbb{E}(X_{1}^{+})^{1+\sigma}<\infty,
$
or if $X_{1}^{+}$ has a tail of index $\alpha\in (2,1+\sigma)$ and $\varepsilon_{n}n^{\rho}\rightarrow0$ as $n\rightarrow\infty$, then
\begin{align}\label{t1.2a}
\lim_{n\rightarrow\infty}\varepsilon^{2\sigma}_{n}n^{\sigma}
\mbb{P}(L_{n}\geq \varepsilon_{n})=\Upsilon(\sigma,\sigma_{0});
\end{align}
(b)~if $X_{1}^{+}$ has a tail of index $\alpha\in (2,1+\sigma)$ and $\varepsilon_{n}n^{\rho}\rightarrow\infty$ as $n\rightarrow\infty$, then
\begin{align}\label{t1.2b}
\lim_{n\rightarrow\infty}\varepsilon^{\alpha}_{n}n^{\alpha-1}
\mbb{P}(L_{n}\geq \varepsilon_{n})=aI(\alpha-1,\sigma);
\end{align}
(c)~if $X_{1}^{+}$ has a tail of index $\alpha\in (2,1+\sigma)$ and $\varepsilon_{n}n^{\rho}\rightarrow \tau\in(0,\infty)$ as $n\rightarrow\infty$, then
\begin{align}\label{t1.2c}
\lim_{n\rightarrow\infty}n^{\frac{\sigma(\alpha-2)}{2\sigma-\alpha}}
\mbb{P}(L_{n}\geq \varepsilon_{n})=\tau^{-2\sigma}\Upsilon(\sigma,\sigma_{0})+
\tau^{-\alpha}aI(\alpha-1,\sigma);
\end{align}
where
\begin{align*}
\Upsilon(\sigma,\sigma_{0})
=\frac{2^{\sigma-1}\Gamma(\sigma+\frac{1}{2})}
{\Gamma(\sigma)\gamma^{\sigma}\sigma\sqrt{\pi}}\sigma_{0}^{2\sigma},
\end{align*}
and $\Gamma(\cdot)$ is the Gamma function.
\end{theorem}

  \begin{remark}

By the total probability formula we have the decomposition:
\begin{align}\label{5.0}
\mbb{P}(L_{n}\geq \varepsilon_{n})
=\sum\limits_{k=1}^{\infty}\mbb{P}(Z_{n}=k)\mbb{P}(S_{k}\geq \varepsilon_{n}k).
\end{align}
We shall split the sum \eqref{5.0} according to $k$ as the following, and estimate the limit behavior of each part and then sum them up:

(a) $k\in (0,\delta/\varepsilon^{2}_{n}],~k\in (\delta/\varepsilon^{2}_{n},M/\varepsilon^{2}_{n}],~
k\in(M/\varepsilon^{2}_{n},\infty)$;

(b) $k\in (0,\delta n],~k\in(\delta n,\infty)$;

(c) $k\in (0,\delta/\varepsilon^{2}_{n}],~k\in (\delta/\varepsilon^{2}_{n},M/\varepsilon^{2}_{n}],~
k\in(M/\varepsilon^{2}_{n},\delta n],~k\in(\delta n,\infty)$,\\
where $\delta\in(0,1)$ and $M\geq 1$.

This decomposition is similar to that of Schr\"{o}der case in \cite{F07} (which is on the supercritical case without immigration). However, differently from \cite{F07}, we need the local probability estimate of the critical GWI $Z_{n}$, which are discussed and  presented by Lemmas~\ref{l2.3}--\ref{l2.5}. We shall see that the converging rate is determined by $\sigma$ (involving the mean of immigration and the variance of branching) and the tail index $\alpha$ of $X_1^+$, while in \cite[Theorem 5]{F07} the rate  depends on the Schr\"{o}der constant and the tail index.
\end{remark}

\begin{theorem}\label{c1.5}
(critical value of $\sigma$) Suppose  condition (A) is satisfied, $\sigma>1$, and $X_{1}^{+}$ has a tail of index $\alpha=1+\sigma$. \\
(a)~If $\varepsilon^{\sigma-1}_{n}\log n\rightarrow0$ as $n\rightarrow\infty$, then
\begin{align}\label{c1}
\lim_{n\rightarrow\infty}\varepsilon^{2\sigma}_{n}n^{\sigma}
\mbb{P}(L_{n}\geq \varepsilon_{n})=\Upsilon(\sigma,\sigma_{0}).
\end{align}
(b) If $\varepsilon^{\sigma-1}_{n}\log n\rightarrow\infty$ as $n\rightarrow\infty$, then
\begin{align}\label{cb}
\lim_{n\rightarrow\infty}\frac{1}{\log n}\varepsilon^{\sigma+1}_{n}n^{\sigma}
\mbb{P}(L_{n}\geq \varepsilon_{n})=aI(\sigma,\sigma).
\end{align}
(c) If $\varepsilon^{\sigma-1}_{n}\log n\rightarrow\tau\in(0,\infty)$ as $n\rightarrow\infty$, then
 \begin{align}\label{c1.7}
\lim_{n\rightarrow\infty}\varepsilon^{2\sigma}_{n}n^{\sigma}
\mbb{P}(L_{n}\geq \varepsilon_{n})=\Upsilon(\sigma,\sigma_{0})+
\tau aI(\sigma,\sigma).
\end{align}
\end{theorem}

  \begin{remark}
For the critical value of $\sigma$, we can see that parts (a) and (c) are similar to those of Theorem \ref{t1.2}, except that the different  scaling of $\varepsilon_n$. In part (b), both the scaling of $\varepsilon_n$ and the converging rate has factor $\log n$, which should be compared with \cite[Corollary 7]{F07} for the Schr\"{o}der case of supercritical branching process, where there is no $\log$ term.
\end{remark}

In the following, we present the case that $\varepsilon_n(\equiv \varepsilon)$ is not varying with $n$, but a fixed positive number.

\begin{corollary}\label{c2}
Suppose  condition (A) is satisfied  and $\sigma>1$.\\
(a)~If
 $X_{1}^{+}$ has a tail of index $\alpha\in(1+\sigma, +\infty)$,
then for any $\varepsilon>0$, there exists $q(\varepsilon)>0$ such that
\begin{align}\label{t1.3a}
\lim_{n\rightarrow\infty}n^{\sigma}
\mbb{P}(L_{n}\geq \varepsilon)=q(\varepsilon)<\infty.
\end{align}
(b)~If $X_{1}^{+}$ has a tail of index $\alpha=1+\sigma$, then
\begin{align}\label{t1.3b}
\lim_{n\rightarrow\infty}\frac{n^{\sigma}}{\log n}
\mbb{P}(L_{n}\geq \varepsilon)=\varepsilon^{-(\sigma+1)}aI( \sigma ,\sigma)<\infty.
\end{align}
(c)~If $X_{1}^{+}$ has a tail of index $\alpha\in (2,1+\sigma)$, then
\begin{align}\label{t1.3c}
\lim_{n\rightarrow\infty}n^{\alpha-1}
\mbb{P}(L_{n}\geq \varepsilon)=\varepsilon^{-\alpha}aI( \alpha-1,\sigma)<\infty.
\end{align}
\end{corollary}

The remainder of the paper is organized as follows. In Section 2, we give some preliminary lemmas and their proofs. Section 3 is devoted to the proofs of the main theorems. As usual, we write $a_{n}\sim b_{n}$, if and only if
\begin{align*}
\lim_{n\rightarrow \infty}\frac{a_{n}}{b_{n}}=1;
\end{align*}
Write $a_{n}=o(b_{n})$ if and only if
\begin{align*}
\lim_{n\rightarrow \infty}\frac{a_{n}}{b_{n}}=0,
\end{align*}
where $\{a_{n},n\geq1\}$ and $\{b_{n},n\geq1\}$ are two sequences of positive numbers.

In the following $\mathbb{N}=\{1,2,3, \cdots\}$, $\mathbb{Z}=\{0,1,2,3, \cdots\}$, and $c, c_{0},c_{1},\cdots$ are positive constants whose value may vary from place to place.

\section{Preliminary results}
\begin{lemma}\label{l2.1} (\cite[Theorem 1]{P72})
Under condition (A), we have
\begin{align}\label{2.1}
\lim\limits_{n\rightarrow \infty}n^{\sigma}H_{n}(x)=U(x),\quad  |x|<1
\end{align}
where $U(x)$ satisfies the functional equation
\begin{align*}
h(x)U(f(x))=U(x).
\end{align*}
The above convergence is uniform over compact subsets of the open unit disc. Denoting the power series representation of $U(x)$ by $\sum_{j=0}^\infty \mu_jx^j$, then
\begin{align}\label{2.2}
\lim_{n\rightarrow\infty} n^{\sigma}\mbb{P}\{Z_{n}=j\}=\mu_{j},\quad j\geq 0.
\end{align}
\end{lemma}
\begin{lemma}\label{l2.2}(\cite[Proposition 4.1, Remark 4.2]{18})
Suppose condition (A) is satisfied. Then, for every $c>0$ there exists constants $c_{1}>0,~c_{2}>0$ such that for each $n\geq 1$ and $0<s\leq cn$,
\begin{align}\label{2.3}
c_{1}(1+\gamma s)^{-\sigma}\leq H_{n}(e^{-\frac{s}{n}})\leq c_{2}(1+\gamma s)^{-\sigma}.
\end{align}
Particularly, if $C(n)>0$ and $C(n)/n\rightarrow0$, as $n\rightarrow\infty$, then for $0<s\leq C(n)$, we have
\begin{align}\label{2.3.1}
H_{n}(e^{-\frac{s}{n}})=(1+\gamma s)^{-\sigma}(1+h(n,s)),
\end{align}
where $h$ is a function on $\mathbb{N}\times [0,\infty)$ such that
\begin{align*}
\lim\limits_{n\rightarrow\infty}\sup\limits_{0< s\leq C(n)}|h(n,s)|=0.
\end{align*}
\end{lemma}

\begin{proof}
\eqref{2.3} is from \cite[Proposition 4.1]{18}. Now we prove \eqref{2.3.1}. By Kesten et. al. \cite[Corollary 1]{Ke66}, as $k\rightarrow\infty$, $$1-f_{k}(0)\sim \frac{1}{\gamma k}, \quad  f_{k+1}(0)-f_{k}(0)\sim \frac{1}{\gamma k^2}. $$  Then
for any $0< s\leq C(n)$, we can define $\ m=m(n,s)$ by
\begin{equation}\label{msbound}
1-f_{m}(0)\geq 1-e^{-\frac{s }{ n}}\geq 1-f_{m+1}(0),
\end{equation}
which implies that
  \begin{equation}\label{msbound3}
\lim\limits_{n\rightarrow\infty}\sup\limits_{0< s\leq C(n)}\bigg|\frac{1-e^{-\frac{s }{ n}}}{1-f_{m}(0)}-1\bigg|=0,
\end{equation}
and there exists $c_3>0$ such that for any $0< s\leq C(n)$, \begin{equation}\label{mmm}
m=m(n,s)\geq  \frac{c_3n}{C(n)}.
\end{equation}
Clearly,
\begin{equation*}
H_{n}(e^{-\frac{s}{n}}) =\prod_{k=0}^{n-1}h(f_{k}(e^{-\frac{s }{ n}
})).
\end{equation*}
By (\ref{msbound}), we obtain that
\begin{equation}\label{Hn1}
(1+\gamma s)^{\sigma}\prod_{k=m}^{n+m-1}h(f_{k}(0))\le (1+\gamma s)^{\sigma}H_{n}(e^{-\frac{s}{n}})\leq
(1+\gamma s)^{\sigma}\prod_{k=m+1}^{n+m}h(f_{k}(0)).
\end{equation}
To get \eqref{2.3.1}, it suffices to prove
\begin{equation}
 \label{midd1}\lim\limits_{n\rightarrow\infty}\sup_{0< s\leq C(n)}\bigg|(1+\gamma s)^{\sigma}\prod_{k=m}^{n+m-1}h(f_{k}(0))-1\bigg|=0,\end{equation}
 and
\begin{equation} \label{midd2}\lim\limits_{n\rightarrow\infty}\sup_{0< s\leq C(n)}\bigg|(1+\gamma s)^{\sigma}\prod_{k=m+1}^{n+m}h(f_{k}(0))-1\bigg|=0.
\end{equation}
 Since the proofs of \eqref{midd1} and \eqref{midd2} are essentially similar, so we prove \eqref{midd1} only.
By the proof of  \cite[Theorem 1] {P72}, there exists a sequence of numbers $\{\nu_n\}$ such that
\begin{equation}\label{frate}
1-f_{n}(0)=\frac{1+\nu_{n}}{\gamma n}\quad \mbox{and}\quad \sum\limits_{n=1}^{\infty}\frac{|\nu_{n}|}{n}<\infty.
\end{equation}
By Taylor's expansion we get
\begin{equation*}
h(x)=e^{-\beta(1-x)+\varphi(x)}
\end{equation*}
as $x\to 1$, where $\varphi(x)\in[c_{4}(x-1)^{2},c_{5}(x-1)^{2}]$. Then, we have
\begin{eqnarray*}
(1+\gamma s)^{\sigma}\prod_{k=m}^{n+m}h(f_{k}(0)) &=&
(1+\gamma s)^{\sigma}\exp \left\{
-\beta \sum_{k=m}^{n+m}\left( 1-f_{k}(0)\right) \right\}\exp \left\{
\sum_{k=m}^{n+m}\varphi(f_{k}(0)) \right\}  \\
&=&(1+\gamma s)^{\sigma} \exp \left\{ -\frac{\beta}{\gamma}\sum_{k=m}^{n+m}\frac{1}{k}\right\}\exp \left\{ -\frac{\beta}{\gamma}\sum_{k=m}^{n+m}\frac{\nu_{k}}{k}\right\}
\exp \left\{
\sum_{k=m}^{n+m}\varphi(f_{k}(0)) \right\}\\
&= &(1+\gamma s)^{\sigma}\exp \left\{ -\sigma\bigg(\ln \frac{n+m}{m}+\varepsilon(n,m)\bigg)\right\}\\
& & \cdot
 \exp \left\{ -\sigma\sum_{k=m}^{n+m}\frac{\nu_{k}}{k}\right\}
\exp \left\{
\sum_{k=m}^{n+m}\varphi(f_{k}(0)) \right\}\\\\
&=& \left(\frac{1+\gamma s}{1+n/m}\right)^{\sigma} I_{0}(n,s),\\
\end{eqnarray*}
where
$$\varepsilon(n,m)=\sum\limits_{k=m}^{n+m}\frac{1}{k}-\ln \frac{n+m}{m},$$
and
\begin{eqnarray*}
I_{0}(n,s):=\exp \left\{ -\sigma\sum_{k=m}^{n+m}\frac{\nu_{k}}{k}\right\}
\exp \left\{
\sum_{k=m}^{n+m}\varphi(f_{k}(0)) \right\}\exp \left\{-\sigma\varepsilon(n,m)\right\}.
\end{eqnarray*}
From (\ref{mmm}) and (\ref{frate}), we have
\begin{eqnarray*}
\lim\limits_{n\rightarrow\infty}\sup_{0< s\leq C(n)}|I_{0}(n,s)-1|=0.
\end{eqnarray*}
Now, to get \eqref{midd1}, we only need to prove
\begin{equation}\label{unibound}
\lim\limits_{n\rightarrow\infty}\sup\limits_{0< s\leq C(n)}\Big|\frac{1+\gamma s}{1+n/m}-1\Big|=0.
\end{equation}
Noticing that
\begin{eqnarray*}
\Big|\frac{1+\gamma s}{1+n/m}-1\Big|=\frac{n}{n+m}\Big|\frac{\gamma ms}{n}-1\Big|\leq \Big|\frac{\gamma ms}{n}-1\Big|,
\end{eqnarray*}
it is sufficient to prove
\begin{equation}\label{unibound1}
\lim\limits_{n\rightarrow\infty}\sup\limits_{0< s\leq C(n)}\Big|\frac{\gamma ms}{n}-1\Big|=0.
\end{equation}
Observe that $$\frac{\gamma ms}{n}=\frac{s/n}{1-e^{-s/n}}\cdot\frac{1-e^{-s/n}}{1-f_{m}(0)}\cdot \gamma m(1-f_{m}(0)). $$
By \eqref{msbound3}, \eqref{mmm} and \eqref{frate}, as $n\to \infty$, each term on the right hand side of above equality converges to $1$,  uniformly in $s$. Hence \eqref{unibound1} holds, and then we obtain \eqref{unibound}. The proof is completed.\qed
\end{proof}
\begin{remark}
(\ref{2.3.1}) was mentioned in a lecture of Professor Vladimir Vatutin and he gave a sketch of the proof in a draft.
\end{remark}

Obviously,
\begin{align}\label{3.0}
J_{n}(r)
=\mbb{E}[Z_{n}^{-r}|Z_{n}>0]=&\frac{1}{\Gamma(r)\mbb{P}(Z_{n}>0)}
\int_{0}^{\infty}\mbb{E}(e^{-tZ_{n}};Z_{n}>0)t^{r-1}dt\\ \nonumber
=&\frac{1}{\Gamma(r)\mbb{P}(Z_{n}>0)}\Big(J_{n1}(r)+J_{n2}(r)\Big),
\end{align}
where
\begin{align*}
&J_{n1}(r):=\int_{0}^{1}\mbb{E}(e^{-tZ_{n}};Z_{n}>0)t^{r-1}dt,\\
&J_{n2}(r):=\int_{1}^{\infty}\mbb{E}(e^{-tZ_{n}};Z_{n}>0)t^{r-1}dt.
\end{align*}
Our next two lemmas give the detailed analysis of $J_{n1}(r)$ and $J_{n2}(r)$.

\begin{lemma}\label{l4.1}
Suppose condition (A) is satisfied.\\
(a) For $r>\sigma$,
\begin{align*}
\lim\limits_{n\rightarrow\infty}n^{\sigma}J_{n1}(r)
=\int_{0}^{1}(U(e^{-u})-U(0))u^{r-1}du<\infty.
\end{align*}
(b) For $r<\sigma$,
\begin{align*}
\lim\limits_{n\rightarrow\infty}n^{r}J_{n1}(r)=
\int_{0}^{\infty}(1+\gamma s)^{-\sigma}s^{r-1}ds<\infty.
\end{align*}
(c) For $r=\sigma$,
\begin{align*}
\lim\limits_{n\rightarrow\infty}\frac{n^{\sigma}}{\log n}J_{n1}(\sigma)=
\gamma^{-\sigma}<\infty.
\end{align*}
\end{lemma}

\begin{proof}
(a) Recalling the definition of $J_{n1}(r)$, we have
\begin{align}
n^{\sigma}J_{n1}(r)\nonumber
=&\int_{0}^{1}n^{\sigma}\mbb{E}(e^{-tZ_{n}};Z_{n}>0)t^{r-1}dt\\\nonumber
=&\int_{0}^{1}n^{\sigma}H_{n}(e^{-t})t^{r-1}dt-
\int_{0}^{1}n^{\sigma}H_{n}(0)t^{r-1}dt\\\label{4.4}
=&R_{1}(n,r)-R_{2}(n,r).
\end{align}
Obviously,
\begin{align}\label{4.5}
\lim\limits_{n\rightarrow\infty} R_{2}(n,r)=
\int_{0}^{1}U(0)t^{r-1}dt<\infty
\end{align}
holds by \eqref{2.1}. Now making a change of variable $s=e^{-t}$, we obtain
\begin{align*}
R_{1}(n,r)
=&\int_{\frac{1}{e}}^{1}n^{\sigma}H_{n}(s)(-\log s)^{r-1}\frac{1}{s}ds.
\end{align*}
Define
\begin{align*}
d_{n}(s)=n^{\sigma}H_{n}(s)(-\log s)^{r-1}\frac{1}{s},\quad  s\in (0,1).
\end{align*}
Then
\begin{align*}
d_{n}(s)\rightarrow d(s):=U(s)(-\log s)^{r-1}\frac{1}{s}.
\end{align*}
From Lemma \ref{l2.2}, we know that there exists a constant $c>0$ such that for each $n\geq 1$ and each $t\in [e^{-1},1)$,
\begin{align*}
H_{n}(t)\leq c(1-\gamma n\log t)^{-\sigma}.
\end{align*}
Hence,
\begin{align*}
d_{n}(s)\leq cn^{\sigma}(1-\gamma n\log s)^{-\sigma}(-\log s)^{r-1}\frac{1}{s}:=g_{n}(s).
\end{align*}
It is not difficult to see that as $n\to\infty$,
\begin{align*}
g_{n}(s)\uparrow c(-\gamma \log s)^{-\sigma}(-\log s)^{r-1}\frac{1}{s}:=g(s),
\end{align*}
and for $r> \sigma$,
\begin{align*}
\int_{\frac{1}{e}}^{1}g(s)ds=c\int_{0}^{1}(\gamma t)^{-\sigma}t^{r-1}dt< \infty.
\end{align*}
Using the dominated convergence theorem, we have
\begin{align}\label{4.7}
\int_{\frac{1}{e}}^{1}d_{n}(s)ds\longrightarrow \int_{\frac{1}{e}}^{1}d(s)ds,\quad n\rightarrow \infty.
\end{align}
By a change of variable $u=-\log s$, the right side of (\ref{4.7}) turns out to be
\begin{align*}
\int_{0}^{1}U(e^{-u})u^{r-1}du<\infty.
\end{align*}
Thus, we have
\begin{align}\label{4.8}
R_{1}(n,r)\rightarrow\int_{0}^{1}U(e^{-u})u^{r-1}du<\infty.
\end{align}
Combing \eqref{4.4}, \eqref{4.5} and \eqref{4.8}, part (a) follows.

(b) Let $t=\frac{s}{n}$. Then,
\begin{align*}
n^{r}J_{n1}(r)
=&\int_{0}^{n}(H_{n}(e^{-\frac{s}{n}})-H_{n}(0))s^{r-1}ds\\
=&\int_{0}^{\infty}I_{(s\leq n)} (H_{n}(e^{-\frac{s}{n}})-H_{n}(0))s^{r-1}ds\\
:=&\int_{0}^{\infty}q_{n}(s)s^{r-1}ds.
\end{align*}
Using  Lemma \ref{l2.2}, there exists $c>0$ such that for each $n\geq1$ and $s>0$,
\begin{align*}
q_{n}(s)\leq c(1+\gamma s)^{-\sigma} :=v(s),
\end{align*}
and
\begin{align*}
q(s):=\lim\limits_{n\rightarrow \infty}q_{n}(s)=(1+\gamma s)^{-\sigma}.
\end{align*}
Clearly, for  $r<\sigma$,
\begin{align*}
\int_{0}^{\infty}v(s)s^{r-1}ds< \infty.
\end{align*}
Therefore, using the dominated convergence theorem, we have
\begin{align*}
\lim\limits_{n\rightarrow\infty}n^{r}J_{n1}(r)=\int_{0}^{\infty}q(s)s^{r-1}ds< \infty.
\end{align*}
(c) First, we give the decomposition
\begin{align}
\frac{n^{\sigma}}{\log n}J_{n1}(\sigma)\nonumber
=&\int_{0}^{(\log n)^{-\frac{1}{2\sigma}}}\frac{n^{\sigma}}{\log n}H_{n}(e^{-t})t^{\sigma-1}dt
+\int_{(\log n)^{-\frac{1}{2\sigma}}}^{1}\frac{n^{\sigma}}{\log n}H_{n}(e^{-t})t^{\sigma-1}dt\\\nonumber
-&\int_{0}^{1}\frac{n^{\sigma}}{\log n}H_{n}(0)t^{\sigma-1}dt\\\label{4.9}
:=&Q_{1}(n,\sigma)+Q_{2}(n,\sigma)-Q_{3}(n,\sigma).
\end{align}
By Lemma \ref{l2.1}, we know that
\begin{align*}
\lim\limits_{n\rightarrow\infty}n^{\sigma}H_{n}(0)=U(0)<\infty,
\end{align*}
hence,
\begin{align}\label{4.10}
\lim\limits_{n\rightarrow\infty}Q_{3}(n,\sigma)=0.
\end{align}
By the monotonicity of $H_{n}(s)$ and Lemma \ref{l2.2},
\begin{align}\label{4.11}
Q_{2}(n,\sigma)\leq c(\log n)^{-\frac{1}{2}}\rightarrow0,\quad as~n\rightarrow\infty.
\end{align}
Finally, we consider $Q_{1}(n,\sigma)$. By changing variables and \eqref{2.3.1},
\begin{align}
Q_{1}(n,\sigma)\nonumber
=&\frac{1}{\log n}\int_{0}^{\frac{n}{(\log n)^{\frac{1}{2\sigma}}}}
H_{n}(e^{-\frac{s}{n}})s^{\sigma-1}ds\\\label{4.12}
=&\frac{1}{\log n}\int_{0}^{\frac{n}{(\log n)^{\frac{1}{2\sigma}}}}
(1+\gamma s)^{-\sigma}s^{\sigma-1}(1+h(n,s))ds.
\end{align}
Noticing that
\begin{align}\label{4.13}
\frac{1}{\log n}\int_{0}^{\frac{n}{(\log n)^{\frac{1}{2\sigma}}}}
(1+\gamma s)^{-\sigma}s^{\sigma-1}ds
\sim\gamma^{-\sigma},
\end{align}
and by Lemma \ref{l2.2},
\begin{align*}
\lim\limits_{n\rightarrow\infty}\sup\limits_{0<s<
\frac{n}{(\log n)^{1/2\sigma}}}|h(n,s)|=0,
\end{align*}
we have
\begin{align}\label{4.14}
\lim\limits_{n\rightarrow\infty}Q_{1}(n,\sigma)=\gamma^{-\sigma}.
\end{align}
Consequently, the assertion of  part (c) follows from \eqref{4.9}-\eqref{4.11} and \eqref{4.14}.\qed
\end{proof}

\begin{lemma}\label{l3.1}
Suppose condition (A) is satisfied. Then\\
(a)\begin{align*}
\lim\limits_{n\rightarrow\infty}n^{r}J_{n2}(r)=0,\quad r<\sigma.
\end{align*}
(b)\begin{align*}
  \lim\limits_{n\rightarrow\infty}n^{\sigma}J_{n2}(r)
=\int_{1}^{\infty}(U(e^{-t})-U(0))t^{r-1}dt<\infty.
\end{align*}

\end{lemma}

\begin{proof}
First, we decompose the integral into two components,
\begin{align}
\nonumber
n^{x}J_{n2}(r)=&\int_{1}^{\infty}n^{x}\mbb{E}(e^{-tZ_{n}} I_{(Z_{n}=1)})t^{r-1}dt +\int_{1}^{\infty}n^{x}\mbb{E}(e^{-tZ_{n}}I_{(Z_{n}\geq 2)})t^{r-1}dt\\ \label{3.3}
=& K_{1}(n,x,r)+K_{2}(n,x,r).
\end{align}
Note that
\begin{align*}
\int_{1}^{\infty}e^{-t}t^{r-1}dt<\infty.
\end{align*}
Together with \eqref{2.2}, we obtain as $n\rightarrow\infty$,
\begin{align}\label{3.4a}
K_{1}(n,\sigma,r)=\int_{1}^{\infty}n^{\sigma}\mbb{P}(Z_{n}=1)e^{-t}t^{r-1}dt
\rightarrow \mu_{1}\int_{1}^{\infty}e^{-t}t^{r-1}dt,
\end{align}
and
\begin{align}\label{3.5}
K_{1}(n,r,r)=\int_{1}^{\infty}n^{r}\mbb{P}(Z_{n}=1)e^{-t}t^{r-1}dt
\rightarrow 0, \quad r<\sigma.
\end{align}
Moreover, by Lemma \ref{l2.1},
\begin{align*}
n^{x}H_{n}(e^{-\frac{1}{2}})\leq n^{\sigma}H_{n}(e^{-\frac{1}{2}})\rightarrow U(e^{-\frac{1}{2}})<\infty,\quad x\leq \sigma,
\end{align*}
as $n\rightarrow\infty$. Therefore,
\begin{align*}
K_{2}(n,x,r)
=& \int_{1}^{\infty}n^{x}\mbb{E}(e^{-t(Z_{n}-1)}I_{(Z_{n}\geq 2)})e^{-t}t^{r-1}dt\\
\leq & \int_{1}^{\infty}n^{x}\mbb{E}(e^{-\frac{t}{2}Z_{n}}I_{(Z_{n}\geq 2)})e^{-t}t^{r-1}dt\\
\leq  &\int_{1}^{\infty}n^{x}H_{n}(e^{-\frac{t}{2}})e^{-t}t^{r-1}dt\\
< & \infty.
\end{align*}
Then
\begin{align}\label{3.7}
\lim\limits_{n\rightarrow \infty}K_{2}(n,r,r)=0,\quad r<\sigma.
\end{align}
 Thus part (a) follows from \eqref{3.5} and \eqref{3.7}.

Note that
\begin{align*}
l(n,\sigma,t,r):=n^{\sigma}\mbb{E}(e^{-tZ_{n}}I_{(Z_{n}\geq 2)})t^{r-1}
\leq n^{\sigma}H_{n}(e^{-\frac{t}{2}})e^{-t}t^{r-1}:=\widetilde{l}(n,\sigma,t,r).
\end{align*}
It follows from \eqref{2.1} that
there exists a constant $c>0$ such that for $t\geq 1$ and $n\geq1$,
\begin{align*}
\widetilde{l}(n,\sigma,t,r)\leq ce^{-t}t^{r-1},
\end{align*}
and
\begin{align*}
\lim\limits_{n\rightarrow \infty}\widetilde{l}(n,\sigma,t,r)
=U(e^{-\frac{t}{2}})e^{-t}t^{r-1}:=\widetilde{l}(t,r).
\end{align*}
Hence, using the dominated convergence theorem,
\begin{align*}
\lim\limits_{n\rightarrow \infty}\int_{1}^{\infty}\widetilde{l}(n,\sigma,t,r)dt=\int_{1}^{\infty}\widetilde{l}(t,r)dt <\infty.
\end{align*}
Thus, using the  modification of  dominated convergence theorem (\cite[Theorem 1.21]{OlavK}), we have
\begin{align}\label{3.10}
\lim\limits_{n\rightarrow \infty}K_{2}(n,\sigma,r)=\int_{1}^{\infty}\lim\limits_{n\rightarrow \infty}l(n,\sigma,t,r)dt.
\end{align}
Applying Lemma~\ref{l2.1} and combining \eqref{3.3}--\eqref{3.4a} with \eqref{3.10}, we obtain part (b).\qed
\end{proof}



\begin{lemma}\label{l2.3}(\cite[Theorem 2.1]{B82})
Suppose condition (A) is satisfied. If there exists a constant $c>0$ such that  $\lim\limits_{n\rightarrow\infty}k_{n}=\infty$ and $\sup\limits_n \frac{k_{n}}{n} \leq c$, then for $k\in[k_{n},cn]$,
\begin{align}\label{2.4}
\mbb{P}(Z_{n}=k)=\frac{1}{\Gamma(\sigma)\gamma^{\sigma}}
\frac{k^{\sigma-1}}{n^{\sigma}}\exp \bigg\{-\frac{k}{\gamma n}\bigg\}(1+\eta_{1}(n,k)),
\end{align}
where $\eta_{1}$ is a function on $\mathbb{N}\times [0,\infty)$ such that
\begin{align*}
\lim_{n\to \infty}\sup\limits_{k\in[k_{n},cn]}
|\eta_{1}(n,k)|=0.
\end{align*}
\end{lemma}

\begin{lemma}\label{l6.1}
Suppose condition (A) is satisfied. Then for every $\varepsilon>0$ such that $\sigma (1-\varepsilon)^{2}-\varepsilon^{2}/\gamma>0$, there exists positive constant $c$ such that for any $n\ge1$ and $|t|\leq \pi/2$,
\begin{align}\label{2.5a}
|H_{n}(e^{it})|\leq c(n|t|)^{-(\sigma (1-\varepsilon)^{2}-\varepsilon^{2}/\gamma)}.
\end{align}
\end{lemma}

\begin{proof}
The proof is essentially  similar  to \cite[Lemmas 6-9]{N06}, with their $f'$ in (2.18)  replaced by our $h$. We omitted the details here.\qed
\end{proof}

\begin{lemma}\label{l2.4}
Suppose condition (A) is satisfied. Then, there exists positive constant $c$ such that for any $n,k\ge1$,
\begin{align}\label{2.5}
\mbb{P}(Z_{n}=k)\leq \frac{c}{k}.
\end{align}
\end{lemma}
\begin{proof}
Suppose $g(s)$ is a probability generating function with mean $M_{0}\in(0,\infty)$ and $g_{n}(s)$ is the $n$th iteration of $g(s)$.  Clearly, for $|s|\leq 1$,
\begin{align}\label{2.6}
|g(s)|\leq 1,\quad |g_{n}(s)|\leq 1,\quad |g'(s)|\leq M_{0}.
\end{align}
Let $X$ be a random variable with characteristic function $\varphi(t)$. The following bound for the concentration function is known from \cite{PetVV75}:  for every $a>0$, it holds that
\begin{align*}
\sup_x\mathbb{P}(X=x)\leq \left(\frac{96}{95}\right)^{2}\frac{1}{a}\int_{-a}^{a}|\varphi(t)|dt.
\end{align*}
Choosing $X$ with generating function $H_{n}'(s)/H_{n}'(1)$ and letting $a=\pi/2$, then there exists a constant $c>0$ such that for any $n$,
\begin{align}\label{2.7}
\sup_k k\mbb{P}(Z_{n}=k)\leq c\int_{-\pi/2}^{\pi/2}|H_{n}'(e^{it})|dt.
\end{align}
The more details can be found in \cite[Lemma 4]{N06}.

Next, we consider the upper bound of $|H_{n}'(e^{it})|$. Observe that for $|s|<1$,
\begin{align*}
H_{n}'(s)=&\sum\limits_{k=0}^{n-1}h'(f_{k}(s))f_{k}'(s)\prod\limits_{j\in \mathbb{Z}\cap [0,n-1]\backslash \{k\}}h(f_{j}(s))\\
=&\sum\limits_{k=0}^{N}h'(f_{k}(s))f_{k}'(s)\prod\limits_{j\in \mathbb{Z}\cap [0,n-1]\backslash \{k\}}h(f_{j}(s))
+H_{n}(s)\sum\limits_{k=N+1}^{n-1}
\frac{h'(f_{k}(s))}{h(f_{k}(s))}f_{k}'(s).
\end{align*}
According to \eqref{2.6},
\begin{align*}
|H_{n}'(s)|
\leq \beta(N+1)
+\Big|H_{n}(s)\sum\limits_{k=N+1}^{n-1}
\frac{h'(f_{k}(s))}{h(f_{k}(s))}f_{k}'(s)\Big|.
\end{align*}
 Setting $s=e^{it}$. Since
\begin{align*}
f_{k}'(s)=f'(f_{k-1}(s))f'(f_{k-2}(s))\cdots f'(s),\quad |s|\le 1,
\end{align*}
and $|f'(s)|\le f'(1)=1$, we obtain $|f_{k}'(e^{it})|\leq 1$. Hence
\begin{align}\label{2.8}
\int_{-\pi/2}^{\pi/2}|H_{n}'(e^{it})|dt\nonumber
\leq &\int_{-\pi/2}^{\pi/2}\beta(1+N)dt
+\int_{|t|<1/n}\sum\limits_{k=N+1}^{n-1}|h'(f_{k}(e^{it}))|\cdot|f_{k}'(e^{it})|\prod\limits_{j\in \mathbb{Z}\cap [0,n-1]\backslash \{k\}}|h(f_{j}(e^{it}))|dt\\
+&\int_{1/n<|t|<\pi/2}\sum\limits_{k=N+1}^{n-1}
|H_{n}(e^{it})|\cdot\frac{|h'(f_{k}(e^{it}))|}{|h(f_{k}(e^{it}))|}\cdot|f_{k}'(e^{it})|dt \nonumber\\
:= &\beta(1+N)\pi+H_1+H_2.
\end{align}
It is not difficult to see that
\begin{align}\label{h1}
H_1\le \frac{2}{n}\beta(n-N)<2\beta.
\end{align}
By the proof of  \cite[Lemma 9]{N06}, for any $\varepsilon>0$, there exists $c>0$ such that for each $k\geq 1$ and $|t|\leq \pi/2$,
\begin{align*}
|f_{k}'(e^{it})|\leq c(k|t|)^{-(2 (1-\varepsilon)^{2}-\varepsilon^{2}/\gamma)}.
\end{align*}
Recalling Lemma \ref{l6.1}, we can choose suitable $\varepsilon$ such that there exists constants $c>0$, $b_2\in (0,1)$,  $b_1\geq1$  and
\begin{align}\label{2.9}
|H_{n}(e^{it})|\leq c(n|t|)^{-b_{1}},\quad
|f_{k}'(e^{it})|\leq c(k|t|)^{-b_{2}},\quad \quad n,k\geq 1,|t|\leq \pi/2.
\end{align}
Moreover, it is known that $f_{n}(e^{it})\rightarrow1$ as $n\rightarrow\infty$. Thus, there exists $N$ such that
\begin{align}\label{2.9.0}
\max\limits_{k\geq N}\frac{1}{|h(f_{k}(e^{it}))|}<\infty.
\end{align}
 Then
\begin{align}\label{h2}
H_2\leq & c\int_{1/n<|t|<\pi/2}(n|t|)^{-b_{1}}\sum\limits_{k=N+1}^{n-1}
(k|t|)^{-b_{2}}dt\\\nonumber
=& cn^{1-b_1-b_2}\int_{1/n<|t|<\pi/2} |t|^{-b_1-b_2} dt
<\infty.
\end{align}
The result is established by \eqref{2.7}--\eqref{h1} and \eqref{h2}.\qed
\end{proof}

{\begin{lemma}\label{l2.5}
Suppose condition (A) is satisfied and $a_{n}\rightarrow\infty$ as $n\rightarrow\infty$. Then, for any $\lambda>0,$ there exists a constant $c_{0}>0$ such that for each $n\geq 1$,
\begin{align}\label{2.10}
\mbb{P}(Z_{n}=k)\leq
\begin{cases}
c_{0}\frac{k^{\sigma+\lambda}}{n^{\sigma}},& 1\leq k\leq a_{n},  \\
c_{0}\frac{k^{\sigma-1}}{n^{\sigma}},& k\geq a_{n}.
\end{cases}
\end{align}
\end{lemma}

\begin{proof}
In fact, for any $a_{n}>0$, we have
\begin{align*}
\sum\limits_{k=1}^{a_{n}}\frac{n^{\sigma}\mbb{P}(Z_{n}=k)}{k^{\sigma+\lambda}}
\leq \sum\limits_{k=1}^{\infty}\frac{n^{\sigma}\mbb{P}(Z_{n}=k)}{k^{\sigma+\lambda}}.
\end{align*}
Then,
\begin{align*}
\overline{\lim\limits_{n\rightarrow\infty}}\sum\limits_{k=1}^{a_{n}}\frac{n^{\sigma}\mbb{P}(Z_{n}=k)}{k^{\sigma+\lambda}}
\leq \lim\limits_{n\rightarrow\infty}\sum\limits_{k=1}^{\infty}\frac{n^{\sigma}\mbb{P}(Z_{n}=k)}{k^{\sigma+\lambda}}
=\lim\limits_{n\rightarrow\infty}n^{\sigma}\mbb{E}(Z_{n}^{-(\sigma+\lambda)}|Z_{n}>0)<\infty.
\end{align*}
The last inequality is given by Theorem \ref{t1.1}. Hence, we can get that
\begin{align*}
\sum\limits_{k=1}^{a_{n}}\frac{n^{\sigma}\mbb{P}(Z_{n}=k)}{k^{\sigma+\lambda}}
\leq c_{0},\quad n\geq 1.
\end{align*}
Then, we have
\begin{align*}
\frac{n^{\sigma}\mbb{P}(Z_{n}=k)}{k^{\sigma+\lambda}}
\leq c_{0},\quad 1\leq k\leq a_{n},~n\geq 1.
\end{align*}
For the case $k\geq a_{n}$, if $a_{n}/n<1$ and $a_{n}\leq k\leq n$, then the result follows from \eqref{2.4}; Otherwise, $a_{n}/n\geq 1$, and then
\begin{align*}
\frac{1}{k}\leq \frac{c}{k}(\frac{k}{n})^{\sigma}=c\frac{k^{\sigma-1}}{n^{\sigma}}.
\end{align*}
Using Lemma \ref{l2.4}, we end the proof.\qed
\end{proof}

\begin{lemma}\label{l2.6}(\cite{N79}Fuk-Nagaev inequality,  or~\cite[Lemma 14]{F07})
For $k\geq 1$, $\varepsilon_{n}>0$, $n\geq 1$, $r>1$ and $t\geq 2$,
\begin{align}\label{2.11}
\mbb{P}(S_{k}\geq \varepsilon_{n}k)\leq k\mbb{P}(X_{1}\geq r^{-1}\varepsilon_{n}k)
+(er\sigma_{0}^{2})^{r}\varepsilon^{-2r}_{n}k^{-r},
\end{align}
and
\begin{align}\label{2.12}
\mbb{P}(S_{k}\geq &\varepsilon_{n}k)\leq k\mbb{P}(X_{1}\geq r^{-1}\varepsilon_{n}k)
+\exp \Big[-\frac{2}{(t+2)^{2}e^{t}\sigma_{0}^{2}}\varepsilon^{2}_{n}k\Big]\\
+ &\Big(\frac{(t+2)r^{t-1}\mbb{E}[X^{t}_{1};0\leq X_{1}\leq \varepsilon_{n}k]}{t\varepsilon^{t}_{n}k^{t-1}}\Big)^{tr/(t+2)}.\nonumber
\end{align}
\end{lemma}

\begin{lemma}\label{l2.10}
Suppose condition (A) is satisfied. Then, there exists $c>0$ such that for each $\delta>0$, we have
\begin{align}\label{sbound}
\limsup\limits_{n\rightarrow\infty}\varepsilon^{2\sigma}_{n}n^{\sigma}\sum\limits_{1\leq k\leq \delta/\varepsilon^{2}_{n}}P(Z_{n}=k)P(S_{k}\geq \varepsilon_{n}k)\leq c\delta^{\sigma}.
\end{align}
\end{lemma}
\begin{proof}
From \eqref{2.10}, we know that
\begin{align}\label{ff}
\sum\limits_{1\leq k\leq a_{n}}P(Z_{n}=k)P(S_{k}\geq \varepsilon_{n}k)
\leq \sum\limits_{1\leq k\leq a_{n}}P(Z_{n}=k)\leq cn^{-\sigma}a_{n}^{1+\sigma+\lambda},
\end{align}
and
\begin{align}\label{SS}
\sum\limits_{a_{n}\leq k\leq \delta/\varepsilon^{2}_{n}}P(Z_{n}=k)P(S_{k}\geq \varepsilon_{n}k)
\leq \sum\limits_{a_{n}\leq k\leq \delta/\varepsilon^{2}_{n}}P(Z_{n}=k)
\leq c\delta^{\sigma}\varepsilon^{-2\sigma}_{n}n^{-\sigma}.
\end{align}
Let $\lambda=\sigma$ and $a_{n}=\varepsilon^{-\frac{\sigma}{2\sigma+1}}_{n}$. We get
\begin{align*}
a_{n}^{1+\sigma+\lambda}=o(\varepsilon^{-2\sigma}_{n}).
\end{align*}
Adding up \eqref{ff} and \eqref{SS} and letting $n\to \infty$, we complete the proof.
\end{proof}\qed

\begin{lemma}\label{l2.7}
Suppose $X_{1}^{+}$ has a tail of index $\alpha>2$. Then, there exists $c>0$ such that for each $n\geq 1$, we have
\begin{align}\label{2.13}
\sum\limits_{k\geq n}\mbb{P}(Z_{n}=k)\mbb{P}(S_{k}\geq \varepsilon_{n}k)\leq
c\Big(\varepsilon^{-\alpha}_{n}n^{1-\alpha}
+\varepsilon^{-2}_{n}n^{-1}\exp\{-c\varepsilon^{2}_{n}n\}\Big).
\end{align}
\end{lemma}

\begin{proof}
Let $t=\alpha+1$ and $r=(t+2)/t$ in \eqref{2.12} and using \eqref{1.6}, we have
\begin{align*}
\mbb{P}(S_{k}\geq \varepsilon_{n}k)\leq
c\varepsilon^{-\alpha}_{n}k^{-(\alpha-1)}+
c\frac{\mbb{E}[X_{1}^{\alpha+1};0\leq X_{1}\leq \varepsilon_{n}k]}
{\varepsilon^{\alpha+1}_{n}k^{\alpha}}+
\exp\{-c\varepsilon^{2}_{n}k\}.
\end{align*}
Note that
\begin{align*}
\mbb{E}[X_{1}^{\alpha+1};0\leq X_{1}\leq x]\leq cx,\quad x\geq 0.
\end{align*}
Hence,
\begin{align*}
\mbb{P}(S_{k}\geq \varepsilon_{n}k)\leq
c\varepsilon^{-\alpha}_{n}k^{-(\alpha-1)}+
\exp\{-c\varepsilon^{2}_{n}k\}.
\end{align*}
Combing this with \eqref{2.5}, we get
\begin{align}\label{2.17}
\sum\limits_{k\geq n}\mbb{P}(Z_{n}=k)\mbb{P}(S_{k}\geq \varepsilon_{n}k)\leq
c\varepsilon^{-\alpha}_{n}\sum\limits_{k\geq n}k^{-\alpha}+
c\sum\limits_{k\geq n}k^{-1}\exp\{-c\varepsilon^{2}_{n}k\}.
\end{align}
Obviously,
\begin{align*}
\varepsilon^{-\alpha}_{n}\sum\limits_{k\geq n}k^{-\alpha}\leq
c\varepsilon^{-\alpha}_{n}n^{1-\alpha},
\end{align*}
and
\begin{align*}
\sum\limits_{k\geq n}k^{-1}\exp\{-c\varepsilon^{2}_{n}k\}\leq
\frac{1}{c\varepsilon^{2}_{n}n}\exp\{-c\varepsilon^{2}_{n}n\}.
\end{align*}
Together with \eqref{2.17}, we complete the proof.
\end{proof}\qed

\begin{lemma}\label{l2.8}
Suppose $X_{1}^{+}$ has a tail of index $\alpha\in(2,1+\sigma)$. If $\varepsilon_{n}\geq n^{-\varrho}$ for some $\varrho\in(0,\frac{1}{2})$, then there exists $c>0$ such that for each $\delta>0$, we have
\begin{align}\label{2.20}
\limsup\limits_{n\uparrow\infty}\Big|\varepsilon^{\alpha}_{n}n^{(\alpha-1)}
\sum\limits_{k\geq \delta n}\mbb{P}(Z_{n}=k)\mbb{P}(S_{k}\geq \varepsilon_{n}k)
-aI(\alpha-1,\sigma)\Big|\leq c\delta^{1+\sigma-\alpha}.
\end{align}
\end{lemma}

\begin{proof}
It is known from \cite [Corollory 6.3] {B00} that for every sequence $b_{k}\rightarrow\infty$,
\begin{align}\label{2.21o}
\mbb{P}(S_{k}\geq x)=(1+\eta_{0}(k,x))k\mbb{P}(X_{1}\geq x),
\end{align}
where $\eta_{0}$ is a function on $\mathbb{N}\times [0,\infty)$ such that
\begin{align}\label{2.21so}
\lim\limits_{k\uparrow\infty}\sup\limits_{x:x\geq b_{k}(k\log k)^{\frac{1}{2}}}|\eta_{0}(k,x)|=0.
\end{align}
Note that for each $n\geq1,~\delta>0$, and $k\geq \delta n$, we have
\begin{align*}
\frac{\varepsilon_{n}k}{(k\log k)^{\frac{1}{2}}}\geq
\delta^{\varrho}\frac{k^{\frac{1}{2}-\varrho}}{(\log k)^{\frac{1}{2}}}\uparrow \infty, \quad as ~ k\uparrow \infty.
\end{align*}
Letting $b_{k}=\delta^{\varrho}\frac{k^{\frac{1}{2}-\varrho}}{(\log k)^{\frac{1}{2}}}$ and using \eqref{2.21o}, we have that as $n\to \infty$,
\begin{align}
\sum\limits_{k\geq \delta n}\mbb{P}(Z_{n}=k)\mbb{P}(S_{k}\geq \varepsilon_{n}k)\nonumber
=&\sum\limits_{k\geq \delta n}(1+\eta_{0}(k,\varepsilon_{n}k))k\mbb{P}(Z_{n}=k)\mbb{P}(X_{1}\geq \varepsilon_{n}k)\\\label{2.23}
\sim & a\varepsilon^{-\alpha}_{n}\sum\limits_{k\geq \delta n}(1+\eta_{0}(k,\varepsilon_{n}k))k^{-(\alpha-1)}\mbb{P}(Z_{n}=k)
\end{align}
as $n\rightarrow\infty$. In addition, choosing $a_{n}=\log n$ and $\lambda=1$ in  \eqref{2.10},  then there exists $c>0$ such that for each $n\geq 1$ and $\delta>0$, we have
\begin{align*}
\sum\limits_{1\leq k\leq \delta n}k^{-(\alpha-1)}\mbb{P}(Z_{n}=k)
&=\sum\limits_{1\leq k\leq \log n}k^{-(\alpha-1)}\mbb{P}(Z_{n}=k)
+\sum\limits_{\log n\leq k\leq \delta n}k^{-(\alpha-1)}\mbb{P}(Z_{n}=k)\\
& \leq c\left((\log n)^{3+\sigma-\alpha}n^{-\sigma}+\delta^{\sigma+1-\alpha}n^{-(\alpha-1)}\right).
\end{align*}
Clearly,
\begin{align*}
(\log n)^{3+\sigma-\alpha}n^{-\sigma}=o(n^{-(\alpha-1)}).
\end{align*}
It follows from  Theorem \ref{t1.1}, for $\alpha-1<\sigma$,
\begin{align*}
\mbb{E}[Z_{n}^{-(\alpha-1)}|Z_{n}>0]\sim n^{-(\alpha-1)}I(\alpha-1,\sigma),\quad n\rightarrow\infty.
\end{align*}
Therefore, there exists $c>0$ such that for each $n\geq 1$ and $\delta>0$,
\begin{align}\label{2.24}
\Big|\sum\limits_{k\geq \delta n}k^{-(\alpha-1)}\mbb{P}(Z_{n}=k)-n^{-(\alpha-1)}I(\alpha-1,\sigma)\Big|
\leq cn^{-(\alpha-1)}\delta^{\sigma+1-\alpha}.
\end{align}
The result is established by \eqref{2.21so}-\eqref{2.24}.
\end{proof}\qed
\begin{lemma}\label{l2.9}
 Suppose $\mbb{E}[X_{1}^{+}]^{1+\sigma}<\infty$. Then, there exists $c>0$ such that for each $M\ge 1$,
\begin{align}\label{2.25}
\lim_{n\rightarrow\infty}\sup\varepsilon^{2\sigma}_{n}n^{\sigma}
\sum\limits_{k\geq M/\varepsilon^{2}_{n}}\mbb{P}(Z_{n}=k)\mbb{P}(S_{k}\geq \varepsilon_{n}k)\leq \frac{c}{M}.
\end{align}
\end{lemma}

\begin{proof}
Combing \eqref{2.10} and \eqref{2.11} with $r=\sigma+1$, we have
\begin{align}\nonumber
&n^{\sigma}\sum\limits_{k\geq M/\varepsilon^{2}_{n}}
\mbb{P}(Z_{n}=k)\mbb{P}(S_{k}\geq \varepsilon_{n}k)\\ \nonumber
&\leq c_{0}\sum\limits_{k\geq M/\varepsilon^{2}_{n}}
k^{\sigma-1}\mbb{P}(S_{k}\geq \varepsilon_{n}k)\\ \label{2.26}
&\leq c_{0}\sum\limits_{k\geq M/\varepsilon^{2}_{n}}k^{\sigma}
\mbb{P}(X_{1}\geq (\sigma+1)^{-1}\varepsilon_{n}k)+
c\varepsilon^{-2(\sigma+1)}_{n}\sum\limits_{k\geq M/\varepsilon^{2}_{n}}k^{-2}.
\end{align}
On the one hand,
\begin{align}\label{2.27}
\varepsilon^{-2(\sigma+1)}_{n}\sum\limits_{k\geq M/\varepsilon^{2}_{n}}k^{-2}\leq \frac{c}{M}\varepsilon^{-2\sigma}_{n};
\end{align}
On the other hand,
\begin{align}\label{2.28}
&\sum\limits_{k\geq M/\varepsilon^{2}_{n}}k^{\sigma}
\mbb{P}(X_{1}\geq (\sigma+1)^{-1}\varepsilon_{n}k)\\ \nonumber
&\leq c\int_{M/\varepsilon^{2}_{n}-1}^{\infty}u^{\sigma}
\mbb{P}(X_{1}\geq (\sigma+1)^{-1}\varepsilon_{n}u)du\\ \nonumber
&\leq c\varepsilon^{-\sigma-1}_{n}\int_{(M-\varepsilon^{2}_{n})/[(\sigma+1)\varepsilon_n]}^{\infty}
v^{\sigma}\mbb{P}(X_{1}\geq v)dv.
\end{align}
By $\mbb{E}[X_{1}^{+}]^{1+\sigma}<\infty$ and $\varepsilon_{n}\rightarrow0$, we obtain
\begin{align*}
\int_{(M-\varepsilon^{2}_{n})/[(\sigma+1)\varepsilon_{n}]}^{\infty}
v^{\sigma}\mbb{P}(X_{1}\geq v)dv\rightarrow0.
\end{align*}
If $\sigma\geq 1$, then the order of \eqref{2.28} is $o(\varepsilon^{-2\sigma}_{n})$, uniformly in $M\geq 1$. Combing \eqref{2.26}-\eqref{2.28}, we complete the proof. If $\sigma<1$, then the assertion can be obtained by \eqref{1.3} and \eqref{1.5}. \qed
\end{proof}

\begin{lemma}\label{l2.11}
Suppose condition (A) is satisfied. Then, for each $0<\delta<1$ and each $M\ge 1$,
\begin{align}\label{2.31}
\lim_{n\rightarrow\infty}\varepsilon^{2\sigma}_{n}n^{\sigma}
\sum\limits_{\delta/\varepsilon^{2}_{n} \leq k\leq M/\varepsilon^{2}_{n}}\mbb{P}(Z_{n}=k)\mbb{P}(S_{k}\geq \varepsilon_{n}k)
=\frac{1}{\Gamma(\sigma)\gamma^{\sigma}}\int_{\delta}^{M}u^{\sigma-1}\Psi(\sqrt{u}/\sigma_{0})du.
\end{align}
\end{lemma}

\begin{proof}
  Denote $$\sum\limits_{n}(\delta,M)=\sum\limits_{\delta/\varepsilon^{2}_{n} \leq k\leq M/\varepsilon^{2}_{n}}\mbb{P}(Z_{n}=k)\mbb{P}(S_{k}\geq \varepsilon_{n}k).$$  Then, by \eqref{2.4},
\begin{align}\label{2.32}
\sum\limits_{n}(\delta,M)=
\sum\limits_{\delta/\varepsilon^{2}_{n} \leq k\leq M/\varepsilon^{2}_{n}}(1+\eta_{1}(n,k))\frac{1}{\Gamma(\sigma)\gamma^{\sigma}}
\frac{k^{\sigma-1}}{n^{\sigma}}\exp \bigg\{-\frac{k}{\gamma n}\bigg\}\mbb{P}(S_{k}\geq \varepsilon_{n}k).
\end{align}
Clearly,
\begin{align}\label{2.33}
&\underline{V}(n)\sum\limits_{\delta/\varepsilon^{2}_{n} \leq k\leq M/\varepsilon^{2}_{n}}
\frac{k^{\sigma-1}}{n^{\sigma}}\mbb{P}(S_{k}\geq \varepsilon_{n}k)\\ \nonumber
&\leq \sum\limits_{\delta/\varepsilon^{2}_{n} \leq k\leq M/\varepsilon^{2}_{n}}
(1+\eta_{1}(n,k))\frac{k^{\sigma-1}}{n^{\sigma}}\exp \bigg\{-\frac{k}{\gamma n}\bigg\}\mbb{P}(S_{k}\geq \varepsilon_{n}k)\\ \nonumber
&\leq \overline{V}(n)\sum\limits_{\delta/\varepsilon^{2}_{n} \leq k\leq M/\varepsilon^{2}_{n}}\frac{k^{\sigma-1}}{n^{\sigma}}\mbb{P}(S_{k}\geq \varepsilon_{n}k),
\end{align}
where
\begin{align*}
\underline{V}(n)=\inf\limits_{\delta/\varepsilon^{2}_{n} \leq u\leq M/\varepsilon^{2}_{n}}(1+\eta_{1}(n,u))
\exp \bigg\{-\frac{u}{\gamma n}\bigg\},
\end{align*}
\begin{align*}
\overline{V}(n)=\sup\limits_{\delta/\varepsilon^{2}_{n} \leq u\leq M/\varepsilon^{2}_{n}}(1+\eta_{1}(n,u))
\exp \bigg\{-\frac{u}{\gamma n}\bigg\}.
\end{align*}
Recalling Lemma \ref{l2.3}, we have
\begin{align}\label{2.34}
\lim_{n\rightarrow\infty}\underline{V}(n)=
\lim_{n\rightarrow\infty}\overline{V}(n)=1.
\end{align}
By the central limit theorem,
\begin{align*}
\mbb{P}(S_{k}\geq \varepsilon_{n}k)=(1+\eta_{2}(n,k))
\Psi(\sqrt{\varepsilon^{2}_{n}k}/\sigma_{0}),
\end{align*}
where $\eta_{2}$ is a function on $\mathbb{N}\times [0,\infty)$ such that
\begin{align}\label{eta2}
\lim_{n\to \infty}\sup\limits_{k\in[\delta/\varepsilon^{2}_{n},M/\varepsilon^{2}_{n}]}
|\eta_{2}(n,k)|=0,
\end{align}
and $1-\Psi(x)$ is the standard normal distribution function. Hence, as $n\rightarrow\infty$,
\begin{align}\label{2.36}
&\sum\limits_{\delta/\varepsilon^{2}_{n} \leq k\leq M/\varepsilon^{2}_{n}}k^{\sigma-1}\mbb{P}(S_{k}\geq \varepsilon_{n}k)\\ \nonumber
&=\sum\limits_{\delta/\varepsilon^{2}_{n} \leq k\leq M/\varepsilon^{2}_{n}}(1+\eta_{2}(n,k))k^{\sigma-1}\Psi(\sqrt{\varepsilon^{2}_{n}k}/\sigma_{0})\\
\nonumber
 &\sim\sum\limits_{\delta/\varepsilon^{2}_{n} \leq k\leq M/\varepsilon^{2}_{n}}\eta_{2}(n,k)k^{\sigma-1}\Psi(\sqrt{\varepsilon^{2}_{n}k}/\sigma_{0})\\
\nonumber
&+\varepsilon^{-2\sigma}_{n}
\int_{\delta}^{M}u^{\sigma-1}\Psi(\sqrt{u}/\sigma_{0})du.
\end{align}
Combing \eqref{2.32}-\eqref{2.36}, we get \eqref{2.31}.\qed
\end{proof}

\begin{lemma}\label{l2.12}
Suppose $X_{1}^{+}$ has a tail of index $\alpha\in(2,1+\sigma)$. Then, there exists $c>0$ such that for each $\delta>0$ and $M\ge 1$, \\
(a) if $\varepsilon_{n} n^{\rho}\rightarrow y\in[0,\infty)$ as $n\rightarrow\infty$, then
\begin{align}\label{2.37}
\limsup\limits_{n\rightarrow\infty}\varepsilon^{2\sigma}_{n}n^{\sigma}\sum\limits_{M/\varepsilon^{2}_{n} \leq k\leq \delta n}\mbb{P}(Z_{n}=k)\mbb{P}(S_{k}\geq \varepsilon_{n}k)\leq
c\delta^{1+\sigma-\alpha}y^{2\sigma-\alpha}
+\frac{c}{M};
\end{align}
(b) if $\varepsilon_{n} n^{\rho}\rightarrow \infty$ as $n\rightarrow\infty$, then
\begin{align}\label{2.38}
\limsup\limits_{n\rightarrow\infty}\varepsilon^{\alpha}_{n}n^{\alpha-1}\sum\limits_{1\leq k\leq \delta n}\mbb{P}(Z_{n}=k)\mbb{P}(S_{k}\geq \varepsilon_{n}k)\leq
c\delta^{1+\sigma-\alpha}.
\end{align}
\end{lemma}

\begin{proof}
Combing \eqref{2.10} and \eqref{2.11} with $r=\sigma+1$, we have
\begin{align}\label{2.39}
&\sum\limits_{M/\varepsilon^{2}_{n}\leq k\leq \delta n}
\mbb{P}(Z_{n}=k)\mbb{P}(S_{k}\geq \varepsilon_{n}k)\\ \nonumber
&\leq cn^{-\sigma}\sum\limits_{M/\varepsilon^{2}_{n}\leq k\leq \delta n}k^{\sigma}
\mbb{P}(X_{1}\geq (\sigma+1)^{-1}\varepsilon_{n}k)+
cn^{-\sigma}\varepsilon^{-2(\sigma+1)}_{n}\sum\limits_{M/\varepsilon^{2}_{n}\leq k\leq \delta n}k^{-2}.
\end{align}
Observe that
\begin{align}\label{2.40}
\varepsilon^{-2(\sigma+1)}_{n}\sum\limits_{M/\varepsilon^{2}_{n}\leq k\leq \delta n}k^{-2}\leq \frac{c}{M}\varepsilon^{-2\sigma}_{n},\quad M\geq 1.
\end{align}
Further, by \eqref{1.6}, we get
\begin{align}\label{2.41}
\sum\limits_{M/\varepsilon^{2}_{n}\leq k\leq \delta n}k^{\sigma}
\mbb{P}(X_{1}\geq (\sigma+1)^{-1}\varepsilon_{n}k)\leq
c\varepsilon^{-\alpha}_{n}\delta^{1+\sigma-\alpha}n^{1+\sigma-\alpha}.
\end{align}
Then, from \eqref{2.39}--\eqref{2.41}  we obtain
\begin{align}\label{x2.37}
\sum\limits_{M/\varepsilon^{2}_{n} \leq k\leq \delta n}\mbb{P}(Z_{n}=k)\mbb{P}(S_{k}\geq \varepsilon_{n}k)\leq
c\Big(\delta^{1+\sigma-\alpha}\varepsilon^{-\alpha}_{n}n^{1-\alpha}
+\frac{1}{M}\varepsilon^{-2\sigma}_{n}n^{-\sigma}\Big).
\end{align}
Considering that $\varepsilon_{n} n^{\rho}\rightarrow y\in[0,\infty)$, we obtain \eqref{2.37}.

Finally, by Lemma \ref{l2.10}, we have
\begin{align*}
\sum\limits_{1\leq k\leq 1/\varepsilon^{2}_{n}}P(Z_{n}=k)P(S_{k}\geq \varepsilon_{n}k)\leq c\varepsilon^{-2\sigma}_{n}n^{-\sigma}.
\end{align*}
Let $M=1$ in \eqref{x2.37}, we get
\begin{align*}
\sum\limits_{1\leq k\leq \delta n}\mbb{P}(Z_{n}=k)\mbb{P}(S_{k}\geq \varepsilon_{n}k)\leq
c\Big(\delta^{1+\sigma-\alpha}\varepsilon^{-\alpha}_{n}n^{1-\alpha}
+\varepsilon^{-2\sigma}_{n}n^{-\sigma}\Big).
\end{align*}
Considering that $\varepsilon_{n} n^{\rho}\rightarrow \infty$, we obtain \eqref{2.38}.
\end{proof}\qed

\section{Proofs of main theorem}

\textbf{Proof of Theorem 1.1.}
We complete the proof by \eqref{3.0} and Lemmas \ref{l4.1} and \ref{l3.1}.

\qed

\textbf{Proof of Theorem 1.2.} (a) From \eqref{5.0}, we have
\begin{align*}
\mbb{P}(L_{n}\geq \varepsilon_{n})
=&\sum\limits_{1\leq k<\delta/\varepsilon^{2}_{n}}\mbb{P}(Z_{n}=k)\mbb{P}(S_{k}\geq \varepsilon_{n}k)\\
+&\sum\limits_{\delta/\varepsilon^{2}_{n}\leq k< M/\varepsilon^{2}_{n}}\mbb{P}(Z_{n}=k)\mbb{P}(S_{k}\geq \varepsilon_{n}k)\\
+&\sum\limits_{k\geq M/\varepsilon^{2}_{n}}\mbb{P}(Z_{n}=k)\mbb{P}(S_{k}\geq \varepsilon_{n}k).
\end{align*}
First, we show that
\begin{align}\label{5.1}
\lim_{n\rightarrow\infty}\sup\varepsilon^{2\sigma}_{n}n^{\sigma}
\sum\limits_{k\geq M/\varepsilon^{2}_{n}}\mbb{P}(Z_{n}=k)\mbb{P}(S_{k}\geq \varepsilon_{n}k)\leq \frac{c}{M},\quad M\geq 1.
\end{align}
Under the condition $\mbb{E}[X_{1}^{+}]^{1+\sigma}<\infty$, this bound follows from Lemma \ref{l2.9}. Hence, we only need to show \eqref{5.1} in the case $X_{1}^{+}$ has a tail of index $\alpha$.

Under the assumption $\varepsilon_nn^\rho\to 0$,
\begin{align*}
\varepsilon^{-\alpha}_{n}n^{1-\alpha}=o(\varepsilon^{-2\sigma}_{n}n^{-\sigma}),
\quad \frac{1}{\varepsilon^{2}_{n}n}\exp\{-c\varepsilon^{2}_{n}n\}
=o(\varepsilon^{-2\sigma}_{n}n^{-\sigma}),\quad n\rightarrow\infty.
\end{align*}
By Lemma \ref{l2.7}, we get
\begin{align*}
\limsup\limits_{n\rightarrow\infty}\varepsilon^{2\sigma}_{n}n^{\sigma}\sum\limits_{k\geq n}\mbb{P}(Z_{n}=k)\mbb{P}(S_{k}\geq \varepsilon_{n}k)
=0.
\end{align*}
Combing with \eqref{2.37} ($\delta=1$, $y=0$), we complete the proof of \eqref{5.1}. Collecting Lemma \ref{l2.10}, Lemma \ref{l2.11}, \eqref{5.1}, and letting $\delta\downarrow0,~M\uparrow\infty$, we obtain  part (a).

(b) In this case, we use the following decomposition:
\begin{align*}
\mbb{P}(L_{n}\geq \varepsilon_{n})
=\sum\limits_{1\leq k\leq \delta n}\mbb{P}(Z_{n}=k)\mbb{P}(S_{k}\geq \varepsilon_{n}k)+\sum\limits_{k\geq \delta n}\mbb{P}(Z_{n}=k)\mbb{P}(S_{k}\geq \varepsilon_{n}k).
\end{align*}
Letting $\delta\downarrow0$ in \eqref{2.38} and \eqref{2.20}, we obtain part (b).

(c) Note that $\varepsilon_{n}n^{\rho}\rightarrow \tau$ as $n\rightarrow\infty$, part (c) follows from \eqref{sbound}, \eqref{2.31}, \eqref{2.37} and  \eqref{2.20}.

\qed

\textbf{Proof of Theorem \ref{c1.5}.}

\eqref{c1} can be proved by the similar way of Theorem \ref{t1.2} (a). For \eqref{cb} and \eqref{c1.7}, differently from Theorem \ref{t1.2}, we  have the following decomposition of $k$: $(0,(\log n)^{\frac{3}{\sigma-1}}],~(
(\log n)^{\frac{3}{\sigma-1}},\infty)$ and
$(0,\delta/\varepsilon^{2}_{n}]$,  $(\delta/\varepsilon^{2}_{n},M/\varepsilon^{2}_{n}]$,
$(M/\varepsilon^{2}_{n},(\log n)^{\frac{3}{\sigma-1}}]$, $((\log n)^{\frac{3}{\sigma-1}},\infty)$, respectively. Furthermore, we need to modify  Lemma \ref{l2.8} as the following: Suppose $X_{1}^{+}$ has a tail of index $\alpha=1+\sigma$. Let $x>0$. If $\varepsilon_{n}\geq (\log n)^{-\varrho}$ for some $\varrho\in(0,\frac{x}{2})$, then
\begin{align}\label{6.1}
\lim\limits_{n\rightarrow\infty}\Big|\frac{1}{\log n}\varepsilon^{\sigma+1}_{n}n^{\sigma}
\sum\limits_{k\geq (\log n)^{x}}\mbb{P}(Z_{n}=k)\mbb{P}(S_{k}\geq \varepsilon_{n}k)
-aI( \sigma,\sigma)\Big|=0.
\end{align}
We omit the details  of the proof here.\qed

\textbf{Proof of Corollary \ref{c2}.}
  We begin with part (a).  First we have
\begin{align}\label{6.2}
n^{\sigma}\mbb{P}(L_{n}\geq \varepsilon)
=\sum\limits_{k=1}^{\infty}n^{\sigma}\mbb{P}(Z_{n}=k)\mbb{P}(S_{k}\geq \varepsilon k):=\sum\limits_{k=1}^{\infty}a_{nk}.
\end{align}
Using Lemma \ref{l2.6} with $r=\alpha-1$,
\begin{align*}
a_{nk}\le &n^{\sigma}\mbb{P}(Z_{n}=k)\left(k\mbb{P}(X_{1}\geq (\sigma+1)^{-1}\varepsilon k)
+ck^{-(\alpha-1)}\right)\\
\leq&cn^{\sigma}\mbb{P}(Z_{n}=k)k^{1-\alpha}\\
:=&cb_{nk}.
\end{align*}
By Lemma~\ref{l2.1}, as $n\to \infty$,
\begin{align*}
a_{nk}\to \widehat{a}_k:=\mu_k \mbb{P}(S_{k}\geq \varepsilon k),
\end{align*}
and
\begin{align*}
b_{nk}\to \widehat{b}_k:=\mu_k k^{1-\alpha}.
\end{align*}
Recalling $\alpha>1+\sigma$ and Theorem \ref{t1.1}, we get as $n\rightarrow\infty$,
\begin{align*}
\sum\limits_{k=1}^{\infty}
b_{nk}
\rightarrow \sum\limits_{k=1}^{\infty}
\widehat{b}_{k}=I(\alpha-1,\sigma)<\infty.
\end{align*}
Now, using the modification of dominated convergence theorem, we get \eqref{t1.3a} with
    $$
    q(\varepsilon)=\sum\limits_{k=1}^{\infty}\mu_k \mbb{P}(S_{k}\geq \varepsilon k).$$

For part (b), in the similar way of Lemma \ref{l2.8} and Lemma \ref{l2.12}, we obtain
\begin{align*}
\lim\limits_{n\rightarrow\infty}\Big|\frac{1}{\log n}\varepsilon^{\sigma+1}n^{\sigma}
\sum\limits_{k\geq \log n}\mbb{P}(Z_{n}=k)\mbb{P}(S_{k}\geq \varepsilon k)
-aI(\sigma,\sigma)\Big|=0,
\end{align*}
and
\begin{align*}
\sum\limits_{1\leq k< \log n}\mbb{P}(Z_{n}=k)\mbb{P}(S_{k}\geq \varepsilon k)\leq
cn^{-\sigma}\log \log n,
\end{align*}
which yield \eqref{t1.3b}.

 For part (c), we also split $k$ into the following two parts: $(0,\log n],~[\log n,\infty)$,  and we can verify that
\begin{align*}
\lim\limits_{n\rightarrow\infty}\Big|\varepsilon^{\alpha}n^{(\alpha-1)}
\sum\limits_{k\geq \log n}\mbb{P}(Z_{n}=k)\mbb{P}(S_{k}\geq \varepsilon k)
-aI( \alpha-1,\sigma)\Big|=0.
\end{align*}
Finally, from \eqref{2.10},
\begin{align*}
\sum\limits_{1\leq k<\log n}
\mbb{P}(Z_{n}=k)\mbb{P}(S_{k}\geq \varepsilon_{n}k)
\leq cn^{-\sigma}(\log n)^{\sigma}=o(n^{-(\alpha-1)}).
\end{align*}
Combining above discussions we obtain \eqref{t1.3c}.}\qed

\end{document}